\documentclass[preprint,12pt]{elsarticle}
\usepackage{amsfonts,fixltx2e,amsmath,graphicx,amssymb,amsthm}
\usepackage{mathrsfs}

\usepackage{times}
\usepackage{amsfonts}
\usepackage{amsthm}
\usepackage{amsmath}
\usepackage{epic}
\usepackage{times}
\usepackage{multirow,array}

\usepackage[margin=3cm]{geometry}
\usepackage{amsthm,amsmath,amssymb}
\usepackage{fullpage}

\usepackage{tikz}
\usetikzlibrary{snakes}
\usepackage{soul}

\newcommand{\nats}{{\mathbb N}}
\newcommand{\ints}{{\mathbb Z}}

\def\A{\mathbb{A}}
\def\G{\mathcal{G}}
\def\Gab{\mathcal{G}_{\mbox{ab}}}

\newtheorem{thm}{Theorem}

\newtheorem{theorem}{Theorem}[section]

\newtheorem{lemma}{Lemma}
\newtheorem{corollary}[thm]{Corollary}
\newtheorem{proposition}[thm]{Proposition}

\newtheorem{remark}[theorem]{Remark}

\theoremstyle{definition}

\begin{document}

\begin{frontmatter}

\title{On a group theoretic generalization of the Morse-Hedlund theorem}

\author[label1]{\'Emilie Charlier}
  \ead{echarlier@ulg.ac.be}

\author[label2]{Svetlana Puzynina\fnref{label4}}
  \ead{s.puzynina@gmail.com}

   \author[label2,label3]{Luca Q. Zamboni}
  \ead{lupastis@gmail.com}

\address[label1]{D\'epartement de Math\'ematique, Universit\'e de Li\`ege, Belgium}
\address[label2]{LIP, ENS de Lyon, Universit\'e de Lyon, France and Sobolev Institute of
Mathematics, Novosibirsk, Russia}
\fntext[label4]{Supported by the LABEX MILYON (ANR-10-LABX-0070) of Universit\'e de Lyon,
within the program ÒInvestissements d'AvenirÓ (ANR-11-IDEX-0007)
operated by the French National Research Agency (ANR).}
\address[label3]{Institut Camille Jordan, Universit\'e Lyon 1, France}

\begin{abstract}
In their 1938 seminal paper on symbolic dynamics, Morse and Hedlund proved
that every aperiodic infinite word $x\in \A^\nats,$ over a non empty finite alphabet $\A,$ contains at least $n+1$ distinct factors of each length $n.$  They further showed that an infinite word $x$ has exactly $n+1$ distinct factors of each length $n$ if and only if $x$ is binary, aperiodic and balanced, i.e., $x$ is a Sturmian word. In this paper we obtain a broad generalization of the Morse-Hedlund theorem via group actions. Given a  subgroup $G$ of the symmetric group $S_n, $ let $1\leq \epsilon(G)\leq n$ denote the number of distinct $G$-orbits of $\{1,2,\ldots ,n\}.$  Since $G$ is a subgroup of $S_n,$ it acts on $\A^n=\{a_1a_2\cdots a_n\,|\,a_i\in \A\}$ by permutation.  Thus, given an infinite word $x\in \A^\nats$ and  an infinite sequence $\omega=(G_n)_{n\geq 1}$ of subgroups $G_n \subseteq S_n,$ we consider the complexity function $p_{\omega ,x}:\nats \rightarrow \nats$  which counts for each length $n$ the number of equivalence classes of factors of $x$ of length $n$ under the action of $G_n.$ We show that if $x$ is aperiodic, then $p_{\omega, x}(n)\geq\epsilon(G_n)+1$ for each $n\geq 1,$ and moreover, if equality holds for each $n,$ then $x$ is Sturmian.  Conversely,  let $x$ be a Sturmian word. Then for every infinite sequence $\omega=(G_n)_{n\geq 1}$ of Abelian subgroups $G_n \subseteq S_n,$   there exists $\omega '=(G_n')_{n\geq 1}$ such that for each $n\geq 1:$ $G_n'\subseteq S_n$ is isomorphic to $G_n$ and $p_{\omega',x}(n)=\epsilon(G'_n)+1.$ 
Applying the above results to the sequence $(Id_n)_{n\geq 1},$ where $Id_n$ is the trivial subgroup of $S_n$ consisting only of the identity, we recover both directions of the Morse-Hedland theorem.
\end{abstract}

\journal{}

\end{frontmatter}


 
 \section{Introduction}

 A celebrated result of Morse and Hedlund states that  every aperiodic (meaning non-ultimately periodic) infinite word $x\in \A^\nats,$ over a non empty finite alphabet $\A,$ contains at least $n+1$ distinct factors of each length $n.$  They further showed that an infinite word $x$ has exactly $n+1$ distinct factors of each length $n$ if and only if $x$ is binary, aperiodic and balanced, i.e., $x$ is a Sturmian word. Thus Sturmian words are those aperiodic words of lowest factor complexity. They arise naturally in many different areas of mathematics including combinatorics, algebra,
number theory, ergodic theory, dynamical systems and differential equations. Sturmian words
also have implications in theoretical physics as $1$-dimensional models of quasi-crystals, and in theoretical computer science where they
are used in computer graphics as digital approximation of straight lines. Despite their simplicity, Sturmian words possess some very deep and mysterious properties   (see \cite{Je1, Je2, Je3, JeAna}). 

There are several variations and extensions of the Morse-Hedlund theorem associated with other notions of complexity including Abelian complexity \cite{CovHed, RSZ}, maximal pattern complexity (introduced by Kamae in \cite{KaZa}),  and palindrome complexity \cite{ABCD} to name just a few. In most cases, these alternative notions of complexity may be used to detect (and in some cases characterize) ultimately periodic words. Generally, amongst all aperiodic
words, Sturmian words  have the lowest possible complexity, although in some cases they are not the only ones (for instance, a restricted class of Toeplitz words is found to have the same maximal pattern complexity as Sturmian words).    There have also been numerous attempts at extending the Morse-Hedlund theorem in higher dimensions. A celebrated conjecture of M. Nivat states that any $2$-dimensional word having at most $mn$ distinct $m\times n$ blocks must be periodic. In this case, it is known that the converse is not true. To this day the Nivat conjecture remains open although the conjecture has been verified for $m$ or $n$ less or equal to $3$ (see \cite{SaTi, CK}).   A very interesting higher dimensional analogue of the Morse-Hedlund theorem was recently obtained by Durand and Rigo in \cite{DuRi} in which they re-interpret the notion of periodicity in terms of Presburger arithmetic.
 
 In this note we give a $1$-dimensional generalization of the Morse-Hedlund theorem in terms of groups actions.  Many naturally occurring dynamical systems come equipped with a group action, however this is not our point of view.  We consider the group action of a permutation group $G\subseteq S_n$ acting on $\A^n,$ the set of words of length $n$ with values in a  finite alphabet $\A.$  In this way, $G$ can only identify elements in the same Abelian class. One advantage to this approach is that it allows us to compare different subshifts relative to the same group action. 
 In order to define a complexity function $p:\nats \rightarrow \nats,$ we consider infinite sequences of permutation groups $\omega= (G_n)_{n\geq 1}$ with each $G_n \subseteq S_n.$  Associated with every such  sequence, and to every infinite word $x\in \A^\nats,$  is a complexity function $p_{\omega ,x}:\nats \rightarrow \nats$  which counts for each length $n$ the number of equivalence classes of factors of $x$ of length $n$ under the action of $G_n.$ 
 We show that if $x$ is aperiodic, then $p_{\omega, x}(n)\geq\epsilon(G_n)+1$ for each $n\geq 1,$  where $\epsilon(G_n)$ is the number of distinct $G$-orbits of $\{1,2,\ldots ,n\}.$ We further show that if equality holds for each $n,$ then $x$ is Sturmian.   In order to obtain a suitable converse, we restrict to sequences of Abelian permutation groups $G_n\subseteq S_n.$  We prove that if $x$ is a Sturmian word, then for every infinite sequence of Abelian permutation groups $(G_n)_{n\geq 1}$ there exists $\omega '=(G_n')_{n\geq 1}$ such that for each $n\geq 1:$ $G_n'\subseteq S_n$ is isomorphic to $G_n$ and $p_{\omega',x}(n)=\epsilon(G'_n)+1.$ This latter result is obtained by associating to each $G_n$ a system of discrete $3$-interval exchange transformations which acts on the factors  of a Sturmian word of length $n.$  Applying these results to the sequence $(Id_n)_{n\geq 1},$ where $Id_n$ is the trivial subgroup of $S_n$ consisting only of the identity, we recover both directions of the Morse-Hedlund theorem.

\section{Main Results}

Let $S_n$ denote the symmetric group on $n$-letters which we regard as the set of all bijections of $\{1,2,\ldots ,n\}.$ Fix a subgroup $G\subseteq S_n.$ We consider the $G$-action $G\times \{1,2,\ldots ,n\}\rightarrow \{1,2,\ldots ,n\}$ given by $g* i=g(i)$ and let $\epsilon(G)$ denote the number of distinct orbits, i.e., \[\epsilon(G)= \mbox{Card}(\{G* i\,|\, i\in \{1,2,\ldots ,n\}\})\]
where $G* i=\{g* i\,|\, g\in G\}$ denotes the $G$-orbit of $i.$ 
For instance if $G$ is the trivial subgroup of $S_n$ consisting only of the identity, then $\epsilon(G)=n,$ while if $G$ contains an $n$-cycle, then $\epsilon(G)=1.$ 
We note $\epsilon(G)$ depends strongly on the embedding of $G$ in $S_n,$ and in fact is not  a group isomorphism invariant, even for isomorphic subgroups of $S_n.$ For instance,
the subgroups $G_1=\{e, (1,2),(3,4),(1,2)(3,4)\}$ and $G_2=\{e, (1,2)(3,4),(1,3)(2,4),(1,4)(2,3)\}$ are two embeddings of the Klein four-group $\ints/2\ints \times \ints/2\ints$ in $S_4,$ and  yet $\epsilon(G_1)=2$ while $\epsilon(G_2)=1.$ On the other hand, it is easily checked that  $\epsilon (G)$ depends only on the conjugacy class of $G$ in $S_n.$  

Let $\A$ be a finite non-empty set. For each $n\geq 1,$  let $\A^n$ denote the set of all words $u=u_1u_2\cdots u_n$ with $u_i\in \A.$ For $a\in \A$ we denote by $|u|_a$ the number of occurrences of the symbol $a$ in $u.$  Two words $u,v\in \A^n$ are {\it Abelian equivalent}, written $u\sim_{ab}v,$ if $|u|_a=|v|_a$ for each $a\in \A.$ 
It is convenient to consider elements of $\A^n$ as functions $u:\{1,2,\ldots ,n\}\rightarrow \A$ where $u(i)=u_i\in \A$ for $1\leq i\leq n.$ For each subset $S\subseteq \{1,2,\ldots ,n\}$ we denote by $u|_S$ the restriction of $u$ to $S.$
There is a natural $G$-action $G\times \A^n\rightarrow \A^n$  given by  $g* u: i\mapsto u( g^{-1}(i))$ for each $i\in \{1,2,\ldots ,n\}.$ In terms of the word representation we have $g* u=u_{g^{-1}(1)}u_{g^{-1}(2)}\cdots u_{g^{-1}(n)}.$
In particular we have $g*u\sim_{ab}u$ for all $g\in G.$ 

Let $x=x_0x_1x_2\cdots \in \A^\nats$ be an infinite word.  Then $G$ defines an equivalence relation $\sim_G$ on  $\mbox{Fact}_x(n)=\{x_ix_{i+1}\cdots x_{i+n-1}\,|\,i\geq 0\},$ the set of factors of $x$ of length $n,$ given by $u\sim_G v$ if and only if $g* u=v$ for some $g\in G,$ in other words if $u$ and $v$ are in the same $G$-orbit relative to the action of $G$ on $\A^n.$ We say the action of $G$ on $\mbox{Fact}_x(n)$ is {\it Abelian transitive} if for all $u,v\in  \mbox{Fact}_x(n)$ we have $u\sim_{ab} v\Leftrightarrow u\sim_G v.$ 

We are interested in the number of distinct $\sim_G$ equivalence classes, i.e., $\mbox{Card}(\mbox{Fact}_x(n)/\sim_G).$ Unlike $\epsilon(G),$ this quantity is not a conjugacy invariant of $G$ in $S_n.$ 
For instance,  consider the  cyclic subgroups $G_1=\langle \sigma_1\rangle$ and $G_2=\langle \sigma_2\rangle$ of $S_4$ where $\sigma_1=(1,2,3,4)$ and $\sigma_2=(1,3,2,4).$ Let $x$ denote the Fibonacci word fixed by the substitution $0\mapsto 01, 1\mapsto 0.$ Then $\mbox{Fact}_x(4)=\{0010, 0100, 0101, 1001, 1010\}$  and
\[\mbox{Fact}_x(4)/\sim_{G_1}=\{[0100\overset{\sigma_1}{\curvearrowright} 0010];[0101\overset{\sigma_1}{\curvearrowright}  1010]; [1001]\}\]
while 
\[\mbox{Fact}_x(4)/\sim_{G_2}=\{[0010\overset{\sigma_2}{\curvearrowright} 0100];[0101\overset{\sigma_2}{\curvearrowright} 1001\overset{\sigma_2}{\curvearrowright} 1010]\}.\]
We observe that the two equivalence classes relative to $\sim_{G_2}$ correspond to the two Abelian classes of
$\mbox{Fact}_x(4).$ Thus the action of $G_2$ on $\mbox{Fact}_x(4)$ is Abelian transitive, while that of $G_1$ is not. 
On the other hand if $y\in \{0,1\}^\nats$ is such that $\mbox{Fact}_y(4)=\{0000, 0001, 0010, 0100, 1000\},$ then the actions of  both $G_1$ and $G_2$ on $\mbox{Fact}_y(4)$ are Abelian transitive.

We apply the above considerations to define a complexity function on infinite words. More precisely, we consider 
the category $\G$ whose objects are all infinite sequences $(G_n)_{n\geq 1}$ where $G_n$ is a subgroup of $S_n$ and where $\mbox{Hom}((G_n)_{n\geq 1}, (G'_n)_{n\geq 1})$ is the collection of all $(f_n)_{n\geq 1}$ where $f_n:G_n\rightarrow G'_n$ is a group homomorphism. 
Two elements $(G_n)_{n\geq 1}, (G'_n)_{n\geq 1} \in \G$ are said to be conjugate if there exists 
$(\sigma_n)_{n\geq 1}$ with $\sigma_n\in S_n$ such that $G'_n=\sigma_n G_n\sigma_n^{-1}$ for each $n\geq 1,$ and isomorphic if there exists $(f_n)_{n\geq 1}\in \mbox{Hom}((G_n)_{n\geq 1}, (G'_n)_{n\geq 1})$ such that 
$f_n:G_n\rightarrow G'_n$ is a group isomorphism for each $n\geq 1.$ 
Associated with every $\omega=(G_n)_{n\geq 1}\in \G$  is a complexity function $p_{\omega ,x}:\nats \rightarrow \nats$  which counts for each length $n$ the number of $\sim_{G_n}$ equivalence classes of factors of $x$ of length $n.$

\begin{thm}\label{GMH1} Let $x\in \A^\nats$ be aperiodic. Then for every infinite sequence $\omega=(G_n)_{n\geq 1}\in \G$ we have $p_{\omega, x}(n) \geq \epsilon(G_n)+1$ for each $n\geq 1.$ Moreover if $p_{\omega, x}(n) =\epsilon(G_n)+1$ for each $n\geq 1,$ then $x$ is Sturmian.
\end{thm}

To obtain a converse to Theorem~\ref{GMH1}, we restrict to the sub category 
$\Gab$ of all infinite sequences $(G_n)_{n\geq 1}$ of Abelian subgroups of $S_n.$
The following constitutes a partial converse to Theorem~\ref{GMH1}:

\begin{thm}\label{GMH2} Let $x$ be a Sturmian word. Then for each infinite sequence  $\omega =(G_n)_{n\geq 1}\in \Gab$  of Abelian permutation groups there exists $\omega'=(G'_n)_{n\geq 1}\in \Gab$ isomorphic to $\omega$ such that for each $n\geq 1$ we have  $p_{\omega',x}(n)=\epsilon(G'_n)+1.$  
\end{thm}

\begin{remark}\rm{Let $\omega=(Id_n)_{n\geq1}\in \Gab,$ where $Id_n$ denotes the trivial subgroup of $S_n$ consisting only of the identity. Then $\epsilon(Id_n)=n$ for each $n\geq 1.$ Moreover, for each infinite word $x,$ we have that $p_{\omega,x}(n)=\mbox{Card}(\mbox{Fact}_x(n)).$   Thus applying Theorem~\ref{GMH1} to $\omega$  we deduce that every aperiodic word $x$ contains at least $n+1$ distinct factors of length $n$ and that if $x$ has exactly $n+1$ distinct factors of each length $n,$ then $x$ is Sturmian. Conversely, if $x$ is Sturmian, then Theorem~\ref{GMH2} applied to 
 $\omega$  implies that $x$ contains exactly  $n+1$ distinct factors of length $n.$ Thus we recover the full Morse-Hedlund theorem.  On the opposite extreme, applying Theorem~\ref{GMH1} to the sequence $\omega=(S_n)_{n\geq 1},$ we recover a result of Coven and Hedlund in \cite{CovHed}.}

\end{remark}

Before embarking on the proofs of Theorems~\ref{GMH1}\&~\ref{GMH2} we review a few basic facts concerning aperiodic words in general and Sturmian words in particular.  For all other definitions and basic notions in combinatorics on words we refer the reader to \cite{Lo, Lo2}.  
 A factor $u$ of an infinite word $x\in \A^\nats$ is called {\it left special} (resp. {\it right special}) if there exist distinct symbols $a,b\in \A$ such that $au$ and $bu$ (resp. $ua$ and $ub$) are factors of $x.$ A factor $u$ which is both left and right special is called {\it bispecial}. If $x$ is aperiodic, then $x$ admits at least one left and one right special factor of each given length. An infinite word $x\in \A^\nats$ is said to be {\it balanced}  if for every pair of factors $u$ and $v$ of $x$ of equal length we have $||u|_a-|v|_a|\leq 1$ for every $a\in \A.$  An infinite word is called  
{\it Sturmian} if it is  aperiodic, binary and  balanced. Equivalently, $x$ is Sturmian if $x$ admits precisely $n+1$ distinct factors of each length $n.$ This implies that $x$ admits exactly one left and one right special factor of each length. Moreover, the set of factors of a Sturmian word is closed under reversal, i.e., $u=u_1u_2\cdots u_n$ is a factor of $x$ if and only if the reverse of $\bar{u}=u_n\cdots u_2u_1$ is a factor of $x$ (see for instance Chapter 2 in \cite{Lo2}). Thus, the right special factors of a Sturmian word are precisely the reversals of the left special factors and vice versa. In particular, bispecial factors of a Sturmian word are palindromes.  Given a factor $u$ of a  Sturmian word $x\in \{0,1\}^\nats$ and $a\in \{0,1\},$ we say  $u$ is {\it rich} in $a$ if $|u|_a\geq |v|_a$ for all factors $v$ of $x$ of length equal to that of $u.$

\begin{proof}[Proof of Theorem~\ref{GMH1}]  We will make use of the following lemma:
\begin{lemma}\label{part} Let $E_1,E_2,\ldots ,E_k$ be a partition of $\{1,2,\ldots ,n\}$ ordered so that $i<j\Rightarrow \max E_i<\max E_j.$ For each $1\leq j\leq k,$ let $\sim_j$ denote the equivalence relation on $\A^n$ defined by $u\sim_j v$ if and only if $u|_{E_i}\sim_{ab} v|_{E_i}$ for each $1\leq i \leq j.$ Then for each aperiodic word $x\in \A^\nats$ and for each $1\leq j\leq k$ we have $\mbox{Card}(\mbox{Fact}_x(n)/\sim_{j})\geq j+1.$
\end{lemma}

\begin{proof} Let $x\in \A^\nats$ be aperiodic. We will show that  $\mbox{Card}(\mbox{Fact}_x(n)/\sim_1)\geq 2 $ and that \[\mbox{Card}(\mbox{Fact}_x(n)/\sim_{j+1})\geq \mbox{Card}(\mbox{Fact}_x(n)/\sim_{j})+1 \] for each $1\leq j\leq k-1.$  Let $m_i=\max E_i.$ Since $x$ is aperiodic, $x$ contains at least one right special factor of each length $m\geq 0.$  In particular, there exists $u\in \A^*,$ with $|u|=m_1-1,$  and distinct letters $a,b\in \A$ such that $ua$ and $ub$ are factors of $x.$ Let $U,V\in \mbox{Fact}_x(n)$ with $ua$ a prefix of $U$ and $ub$ a prefix of $V.$ As 
$ua\nsim_{ab} ub$ and $|ua|=|ub|=m_1\in E_1,$ we have $U|_{E_1}\nsim_{ab}V|_{E_1},$ whence
$U\nsim_1 V.$ Thus $\mbox{Card}(\mbox{Fact}_x(n)/\sim_1)\geq 2 .$
Next fix $1\leq j\leq k-1.$ We will show the existence of two factor $U$ and $V$ of length $n$ such that
$U\sim_j V$ and $U\nsim_{j+1}V.$ As above, since $x$ is aperiodic, there exists $u\in \A^+,$ with $|u|=m_{j+1}-1,$  and distinct letters $a,b\in \A$ such that $ua$ and $ub$ are factors of $x.$ Let $U,V\in \mbox{Fact}_x(n)$ with $ua$ a prefix of $U$ and $ub$ a prefix of $V.$ Then for each $1\leq i\leq j$ we have $U|_{E_i}=u|_{E_i}=V|_{E_i}$ and hence $U\sim_j V.$ On the other hand, as before, since $ua\nsim_{ab} ub$ and $|ua|=m_{j+1},$ we have $U\nsim_{j+1} V.$\end{proof}

Fix $n\geq 1,$ and put $G=G_n$ and $\epsilon(G)=k.$ We will show that if $x\in \A^\nats$ is aperiodic, then $\mbox{Card}(\mbox{Fact}_x(n)/\sim_G)\geq k+1. $ Let $E_1,E_2,\ldots ,E_k$ denote the full set of $G$-orbits of $\{1,2, \ldots ,n\}.$ Then $E_1,E_2,\ldots ,E_k$ is a partition of $\{1,2,\ldots ,n\}.$
For $1\leq j\leq k$ let  $\sim_j$ denote the equivalence relation on $\mbox{Fact}_x(n)$ defined in the previous lemma. Then  for all $u,v\in \mbox{Fact}_x(n)$ we have $u\sim_G v$ implies $u\sim_k v.$ Thus
\[  \mbox{Card}(\mbox{Fact}_x(n)/\sim_G)\geq  \mbox{Card}(\mbox{Fact}_x(n)/\sim_k)\geq k+1=\epsilon(G)+1\]
as required. 
This concludes our proof of the first statement of Theorem~\ref{GMH1}. 

Next suppose that $p_{\omega, x}(n) =\epsilon(G_n)+1$ for each $n\geq 1.$ We will show that $x$ is binary and balanced. Since $x$ is already assumed aperiodic, it will follow that $x$ is Sturmian. Since $\epsilon(G_1)=1,$ and hence $p_{\omega, x}(1) =2,$ it follows that $x$ is on a binary alphabet which we can take to be $\{0,1\}.$

\begin{lemma} Let $x\in \{0,1\}^\nats$ be aperiodic. Then either $x$ is Sturmian or there exist a positive integer $n\geq 2,$ $u\in \{0,1\}^{n-2}$ and a Sturmian word $y$ such that $\mbox{Fact}_x(n)=\mbox{Fact}_y(n)\cup \{0u0, 1u1\}.$
\end{lemma}

For convenience, we will make us of the following notation: Given $u$ and $v$ factors of $x$ with $u$ a prefix of $v,$ we write $u\models_xv$ to mean that each occurrence of $u$ in $x$ is an occurrence of $v.$ Clearly, if $u\models_xv$ and $u$ is both a proper prefix and a proper suffix of $v,$ then $x$ is ultimately periodic. 

\begin{proof} Suppose $x$ is not Sturmian. Then there exists a least positive integer $n\geq 2$ such that 
for all Sturmian words $z$ we have $\mbox{Fact}_x(n)\neq \mbox{Fact}_z(n).$ By minimality of $n$ there exists a Sturmian word $y$ such that $\mbox{Fact}_x(n-1)= \mbox{Fact}_y(n-1).$ It follows that there exists a factor $u\in \mbox{Fact}_x(n-2)=\mbox{Fact}_y(n-2)$ which is bispecial in both $x$ and $y.$  In fact, let $u$ be the unique right special factor of $x$ and  $y$ of length $n-2.$ If $u$ is not left special, then there exists a unique factor 
$v\in \mbox{Fact}_x(n-1)=\mbox{Fact}_y(n-1)$  ending in $u,$ and this factor would necessarily be right special in both $x$ and $y.$ Moreover all other factors of $x$ and $y$ of length $n-1$ admit a unique extension to a factor of length $n$ determined by  their suffix of length $n-2.$ Hence we would have
$\mbox{Fact}_x(n)= \mbox{Fact}_y(n)$ contrary to the choice of $n.$ Thus $u$ is also left special (in both $x$ and $y)$ and hence bispecial.

Since $x$ is aperiodic, at least one of $0u$ or $1u$ is right special in $x.$ Without loss of generality we may assume $0u$ is right special. We now claim that $1u$ must also be right special.   In fact, suppose to the contrary that $1u\models_x1ua$ for some $a\in \{0,1\}.$ If $a=0,$ then $\mbox{Fact}_x(n)$ would coincide with the set of factors of length $n$ of some Sturmian word, contrary to our choice of $n.$ Thus $a=1.$ We will show that this implies that $x$ is ultimately periodic, and hence gives rise to a contradiction. We consider two cases: First suppose no non-empty  prefix of $1u$ is right special; in this case  $1\models_x1u\models_x1u1$ whence $x$ is ultimately periodic. Thus  we may assume that some prefix $1v$ of $1u$ is right special. Consider the longest such right special prefix  $1v.$ Since we are assuming that $1u$ is not right special, it follows that $1vb$ is a prefix of $1u$ for some $b\in \{0,1\}.$  Since $vb$ is left special (as $vb$ is a prefix of $u)$, and since $1v$ is right special, we deduce that $vb$ is equal to the reverse of $1v$ from which it follows that  $b=1.$ Thus as $1v$ is a suffix of $1u,$ we have  $1v1$ is a suffix of $1u1.$ Now since  $1v1\models_x1u\models_x1u1$ and $1v1$ is a proper suffix of $1u1,$ it follows that $x$ is ultimately periodic. Thus we have shown that both $0u$ and $1u$ are right special. Since in $y$ exactly one of $0u$ and $1u$ is right special, the result follows. \end{proof}

Returning to the proof of Theorem~\ref{GMH1}, let us suppose that $p_{\omega, x}(n) =\epsilon(G_n)+1$ for each $n\geq 1$ and that $x$ is not Sturmian. By the previous lemma there exist a positive integer $n\geq 2,$ $u\in \{0,1\}^{n-2}$ and a Sturmian word $y$ such that $\mbox{Fact}_x(n)=\mbox{Fact}_y(n)\cup \{0u0, 1u1\}.$ Since $y$ is Sturmian, at most one of $ \{0u0, 1u1\}$ is a factor of $y.$
Thus by the first part of Theorem~\ref{GMH1}  applied to the aperiodic word $y,$ we deduce that $p_{\omega, x}(n)\geq p_{\omega, y}(n) +1\geq \epsilon(G_n)+2,$ a contradiction. This concludes our proof of Theorem~\ref{GMH1}\end{proof}

We next establish various lemmas leading up to the proof of Theorem~\ref{GMH2}. 
As is well known, every finite Abelian group $G$ can be written multiplicatively as a direct product of cyclic groups $\ints /m_1\ints\times \ints /m_2\ints\times \cdots \times \ints /m_k\ints$ where the $m_i$ are  prime powers.
The unordered sequence $(m_1,m_2,\ldots ,m_k)$ completely determines $G$ up to isomorphism and any symmetric function of the $m_i$  is an isomorphic invariant of $G.$ We consider the {\it trace} of $G$ given by 
$T(G)=m_1+m_2+\cdots +m_k,$ and recall the following result from \cite{Hoff}.

\begin{proposition}\label{trace} If an Abelian group $G$ is embedded in $S_n,$ then $T(G)\leq n.$ 
\end{proposition}

A partition $\{E_1,E_2,\ldots ,E_k\}$ of $\{1,2,\ldots ,n\}$ is called an {\it interval partition} if  for each $1\leq r<s\leq n,$ we have $r,s\in E_i \Rightarrow t\in E_i$ for all $r\leq t \leq  s.$

\begin{lemma}\label{intpart}Let  $\{E_1,E_2,\ldots ,E_k\}$ be an interval partition of $\{1,2,\ldots ,n\}$ ordered so that
$i<j\Rightarrow \max E_i<\max E_j.$ For each $1\leq j\leq k,$ let $\sim_j$ denote the equivalence relation on $\A^n$ defined by $u\sim_j v$ if and only if $u|_{E_i}\sim_{ab} v|_{E_i}$ for each $1\leq i \leq j.$ Then for each Sturmian word $x\in \{0,1\}^\nats$  we have $\mbox{Card}(\mbox{Fact}_x(n)/\sim_{j})= j+1$  for each $1\leq j\leq k.$
\end{lemma}

\begin{proof} Let $x\in \{0,1\}^\nats$ be a Sturmian word. In view of  Lemma~\ref{part} it suffices to show that $\mbox{Card}(\mbox{Fact}_x(n)/\sim_{j})\leq  j+1$  for each $1\leq j\leq k.$ Since $x$ is Sturmian, there are exactly two Abelian classes of factors of $x$ of each length $m,$ thus $\mbox{Card}(\mbox{Fact}_x(n)/\sim_{1})=  2.$ It also follows from this that for  $1\leq j\leq k-1,$ each $\sim_j$ class splits into at most two $\sim_{j+1}$ classes. So it suffices to show that for each  $1\leq j\leq k-1,$ at most one $\sim_j$ class splits under $\sim_{j+1}.$ So fix $1\leq j \leq k-1,$ and suppose to the contrary that two distinct $\sim_j$ classes split under $\sim_{j+1}.$  Then, there exist $u,u',v,v'\in \mbox{Fact}_x(n)$ such that $u\sim_j u', \,v\sim_jv',\,u\nsim_j v,\,u\nsim_{j+1}u'$ and $v\nsim_{j+1}v'.$  Exchanging if necessary  $u$ and $u'$ and/or $v$ and $v',$  we may assume $u|_{E_{j+1}}$ and $v|_{E_{j+1}}$  are rich in $0$ while $u'|_{E_{j+1}}$ and $v'|_{E_{j+1}}$ are rich in $1.$ 
Since $u\nsim_j v,$ there exists a largest integer $1\leq i\leq j$ such that $u|_{E_i}\nsim_{ab}v|_{E_i}.$
Exchanging if necessary $u$ and $v$ and $u'$ and $v',$ we may assume that $u|_{E_i}$ is rich in $0$ and $v|_{E_i}$ is rich in $1.$ 
Since $v|_{E_i}\sim_{ab}v'|_{E_i},$ we have that
$u|_{E_i\cup \cdots \cup E_{j+1}}$ has two more occurrences of $0$ than $v'|_{E_i\cup \cdots \cup E_{j+1}},$ contradicting that $x$ is balanced. \end{proof}

In the next lemma we consider a discrete $3$-interval exchange transformation $(a,b,c)$ defined on the set  $\{1,2,\ldots ,n\}$ (where $n=a+b+c)$ in which the numbers $1,2,\ldots ,n$ are divided into three subintervals of length $c,b$ and $a$ respectively which are then rearranged in the order $a,b,c.$ In other words 
\[1,2,\ldots ,n \mapsto a + b + 1, a + b + 2,\ldots , n,  a + 1, a + 2,\ldots , a + b, 1, 2,\ldots ,a.\]
This is also called an $abc$-permutation in \cite{Pak}. We also include here the degenerate case in which one of $a,b$ or $c$ equals $0.$ 
The following proposition asserts that for each Sturmian word $x$ and for each positive integer $m,$ there exists a $m$-cycle corresponding to a discrete $3$-interval exchange transformation which identifies all factors of $x$ of length $m$ belonging to the same Abelian class.

\begin{lemma}\label{array cycle}Let $x\in \{0,1\}^\nats$ be a Sturmian word. Then for each positive integer $m$ there exists 
a discrete $3$-interval exchange transformation $(a,b,c)$ on $\{1,2,\ldots ,m\}$ given by a $m$-cycle $\sigma$ such that the action of $\langle \sigma \rangle$ on $\mbox{Fact}_x(m)$ is Abelian transitive. \end{lemma} 

\begin{proof} The result is immediate in case $m=1,2,$ or $3.$ In fact, in this case we may take $\sigma=\mbox{id}, $  $(1,2),$ or $(1,2,3)$ respectively. Thus we assume $m\geq 4.$ Let $w$ and $w'$  be two consecutive bispecial factor of $x$ such that $|w'|+2<m\leq |w|+2.$ Let $r$ and $s$ denote the number of occurrences of $1$ and $0$ in $0w 1,$ i.e.,  $r=|0w 1|_1,$ $s=|0w 1|_0,$ so that $r+s=|w|+2.$ Set $p=r^{-1}\bmod {(r+s)}$ and $q=s^{-1}\bmod {(r+s)}.$ Then  $p$ and $q$ are the relatively prime Fine and Wilf periods of the central Sturmian word $w$ (see Proposition 2.1 in \cite{BLR}). Set $a=m-p,$ $b=p+q-m,$ $c=m-q$  and let $\sigma \in S_m$ denote the corresponding
$abc$-permutation. We note that $|w'|+2=\max\{p,q\},$ whence $a$ and $c$ are both positive while $b\geq 0.$ Since $\gcd (a+b,b+c)=\gcd(q,p)=1,$ it follows from Lemma~1 of \cite{Pak} that $\sigma$ is a $m$-cycle. 

Now let $u$ and $v$ be two lexicographically consecutive factors of $x$ of length $m$ with $u<v.$ Assume further that $u$ and $v$ are in the same Abelian class. We will show that $v=\sigma * u.$ 
We consider the lexicographic Christoffel array  $\mathcal{C}_{r,s}$ in which the cyclic conjugates of $0w1$ are ordered lexicographically  in a rectangular array (see \cite{JeZa}). 
For instance, if $w=010010,$ the corresponding Christoffel array  $\mathcal{C}_{3,5}$ is shown in Figure \ref{fig:matrix}. 
Let $U$ and $V$ be two lexicographically consecutive factors of $x$ of length $|w|+2$ with $u$  a prefix of $U$ and $v$ a prefix of $V.$ We recall that $U$ and $V$ differ in exactly two positions, more precisely we can write $U=X01Y$ and $V=X10Y$ for some $X,Y\in \{0,1\}^*$ (see Corollary 5.1 in \cite{BoRe}).
Writing $U=CBAB'$ where $|A|=a,$ $|B|=|B'|=b$ and $|C|=c.$ 
By Theorem C in \cite{JeZa} we have that $V$ is obtained by $U$ by a cyclic shift corresponding to $q,$ i.e., $V=AB'CB.$ In case $m=|w|+2$ so that $u=U$ and $v=V,$ then $B$ and $B'$ are empty and we see that $V=\sigma * U.$ Otherwise, in case $m<|w|+2,$ since $u$ and $v$ are distinct and belong to the same Abelian class, we have that $X01$ is a prefix of $u=CBA$ which in turn implies that $B=B'.$ Whence $U=CBAB$ and $V=ABCB$ and hence $u=CBA$ and $v=ABC$ and $v=\sigma * u$ as required. \end{proof}

\begin{figure}

\begin{center}
\newcolumntype{C}{>{\centering\arraybackslash}b{.1mm}<{}} $$
\mathcal{C}_{3,5}=
\left(
{\begin{tabular}{*{10}{C}}
 0 & 0 & 1 & 0 & 0 & 1 & 0 & 1\\
 0 & 0 & 1 & 0 & 1 & 0 & 0 & 1\\
 0 & 1 & 0 & 0 & 1 & 0 & 0 & 1\\
 0 & 1 & 0 & 0 & 1 & 0 & 1 & 0\\
 0 & 1 & 0 & 1 & 0 & 0 & 1 & 0\\
 1 & 0 & 0 & 1 & 0 & 0 & 1 & 0\\
 1 & 0 & 0 & 1 & 0 & 1 & 0 & 0\\
 1 & 0 & 1 & 0 & 0 & 1 & 0 & 0
\end{tabular}}
\hspace{1.5mm}\right) $$
\end{center}
\caption{The Christoffel array $ \mathcal{C}_{5,3}$.
\label{fig:matrix}}
\end{figure}

\noindent As an immediate consequence of Lemma~\ref{array cycle} we have

\begin{corollary}Let $x\in \{0,1\}^\nats$ be a Sturmian word. Then for each positive integer $n$ there exists a cyclic group $G_n$ generated by an $n$-cycle such that $\mbox{Card}(\mbox{Fact}_x(n)/\sim_{G_n})=2.$
\end{corollary}

\noindent In contrast, if we set $G_n=\langle (1,2,\ldots ,n)\rangle,$  then $\limsup_{n\rightarrow \infty} \mbox{Card}(\mbox{Fact}_x(n)/\sim_{G_n})=+\infty$ (see Theorem~1 of \cite{CFSZ}),  while  $\liminf_{n\rightarrow \infty} \mbox{Card}(\mbox{Fact}_x(n)/\sim_{G_n})=2$ (see  Lemma~9 of \cite{CFSZ}).  

\noindent As another consequence of Lemma~\ref{array cycle} we have:

\begin{lemma}\label{multi array} Let $x\in \{0,1\}^\nats$ be a Sturmian word. Let $\{E_1,E_2,\ldots ,E_k\}$ be an interval partition of $\{1,2,\ldots ,n\},$ and put $m_i=\mbox{Card}(E_i).$  Then there exist cycles $\sigma_1,\sigma_2, \ldots ,\sigma_k$ such that for each $1\leq i\leq k$ we have  $\sigma_i=(a_1,a_2,\ldots ,a_{m_i})$ where $E_i=\{a_1,a_2,\ldots ,a_{m_i}\}.$ Moreover, if $G$ denotes the subgroup of $S_n$ generated by $\sigma_1,\sigma_2, \ldots ,\sigma_k,$ then for all factors $u,v\in \mbox{Fact}_x(n)$ we have $u\sim_G v$ if and only if $u|_{E_i}\sim_{ab} v|_{E_i}$ for each $1\leq i\leq k.$ 
\end{lemma}

\begin{proof} By Lemma~\ref{array cycle}, for each $i$ there exists a cycle $\sigma_i=(a_1,a_2,\ldots ,a_{m_i})$ with $E_i=\{a_1,a_2,\ldots ,a_{m_i}\},$ such that for all factors $u,v\in \mbox{Fact}_x(n)$ we have $u\sim_{\langle \sigma_i\rangle} v$ if and only if $u|_{E_i}\sim_{ab} v|_{E_i}.$ In fact, $\{u|_{E_i}\,|\, u\in \mbox{Fact}_x(n)\}=\mbox{Fact}_x(m_i).$ Moreover as the $E_i$ are pairwise disjoint, the same is true of the $\sigma_i.$ Hence the $\sigma_i$ commute with one another. Thus, given $u, v\in \mbox{Fact}_x(n),$ if $u\sim_G v,$ then there exists $g=\sigma_1^{r_1}\cdots \sigma_k^{r_k}\in G$ such that $v=g* u.$
However, for each $1\leq i\leq k$ we have  $(g* u)|_{E_i}=(\sigma_i^{r_i}* u)|_{E_i},$ hence $u|_{E_i}\sim_{ab} v|_{E_i}.$ Conversely if $u|_{E_i}\sim_{ab} v|_{E_i}$ for each $1\leq i\leq k,$ there exists $r_i$ such that $v|_{E_i}=(\sigma_i^{r_i}* u)|_{E_i}.$ Hence setting $g= \sigma_1^{r_1}\cdots \sigma_k^{r_k}\in G$ we have $v=g* u.$\end{proof}

\noindent We now prove Theorem~\ref{GMH2}.

\begin{proof}[Proof of Theorem~\ref{GMH2}] Let $x\in \{0,1\}^\nats$ be a Sturmian word and let $(G_n)_{n\geq 1}$ be a sequence of Abelian permutation groups. We show that for each $n\geq 1$ there exists a permutation group $G_n'\subseteq S_n$ isomorphic to $G$ such that $\mbox{Card}(\mbox{Fact}_x(n)/\sim_{G'_n})=\epsilon(G'_n)+1.$
Fix $n\geq 1$ and put $G=G_n.$ By the fundamental theorem of finite Abelian groups, $G$ is isomorphic to a direct product  $\ints /m_1\ints\times \ints /m_2\ints\times \cdots \times \ints /m_k\ints$ where the $m_i$ are prime powers. Let $m=T(G)=m_1+m_2+\cdots +m_k.$ By Proposition~\ref{trace} we have $m\leq n.$ Thus, short of adding additional copies of the trivial cyclic group $\ints/1\ints$ or order $1,$ we may assume that $T(G)=n.$ 
Let $E_1=\{1,2,\ldots ,m_1\}, \, E_2=\{m_1+1, \ldots ,m_1+m_2\}, \ldots ,E_k=\{m_1+\cdots m_{k-1}+1,\ldots ,n\}.$ 
Then $\{E_1,E_2,\ldots ,E_k\}$ is an interval partition of $\{1,2,\ldots ,n\}.$ Pick cycles $\sigma_1,\sigma_2,\ldots ,\sigma_k$ as in Lemma~\ref{multi array}. Then the $\sigma_i$ are pairwise disjoint (and hence commute with one another) and each $\sigma_i$ is of order $m_i.$ Hence, the subgroup $G'$ of $S_n$ generated by $\sigma_1,\sigma_2,\ldots ,\sigma_k$ is isomorphic to $G.$
Moreover, $E_1,E_2,\ldots ,E_k$ is the full set of $G'$-orbits of $\{1,2,\ldots n\}$ whence $\epsilon(G')=k.$  
Also by Lemma~\ref{multi array}, for all $u, v\in \mbox{Fact}_x(n)$  we have that $u\sim_{G'}v$ if and only if $u|_{E_i}\sim_{ab} v|_{E_i}$ for each $1\leq i\leq k.$ Thus the equivalence relation $\sim_{G'}$ on $\mbox{Fact}_x(n)$ coincides with the equivalence relation $\sim_k$ given in Lemma~\ref{intpart}. Thus by Lemma~\ref{intpart} we deduce that
\[\mbox{Card}(\mbox{Fact}_x(n)/\sim_{G'}) =\mbox{Card}(\mbox{Fact}_x(n)/\sim_k)=k+1=\epsilon(G')+1\]
as required. This concludes our proof of Theorem~\ref{GMH2}. \end{proof}

\noindent As an immediate consequence of Theorem~\ref{GMH2} and Cayley's theorem we have

\begin{corollary}Let $G$ be an Abelian group of order $n.$ Then for every Sturmian word $x$ there exists a permutation group $G'\subseteq S_n$  isomorphic to $G$ such that $\mbox{Card}(\mbox{Fact}_x(n)/\sim_{G'})=\epsilon(G')+1.$
\end{corollary}

The following example illustrates that in Theorem~\ref{GMH2}, we cannot replace ``isomorphic" by ``conjugate".  Let $G$ be the cyclic subgroup of order $3$ of $S_6$ generated by the permutation $\sigma=(1,2,3)(4,5,6).$ 
Then $\epsilon (G)=2.$ 
We will show that if $x$ is the Fibonacci word, then \[\mbox{Card}\left(\mbox{Fact}_x(6)/\sim_{G'}\right)\geq 4\] for each subgroup $G'$ of $S_6$ conjugate to $G.$ 
To see this, let $G'\subseteq S_6$ be generated by the permutation $(a,b,c)(d,e,f)$ where $\{a,b,c,d,e,f\}=\{1,2,3,4,5,6\}.$    We claim that $100101$ and $101001$ belong to distinct equivalence classes under the action of $G'$ on $\{0,1\}^6.$ In fact, suppose to the contrary that $g(100101)=101001$ for some $g\in G'.$ Then $g(\{1,4,6\})=\{1,3,6\}$ and $g(\{2,3,5\})=\{2,4,5\}.$
We claim that either $g(4)=3$ or $g(3)=4.$ Otherwise, $g: 4\mapsto x\mapsto y$ where $\{x,y\}=\{1,6\}.$
But $4\notin g(\{1,6\}).$ Thus without loss of generality we can assume $g(4)=3.$ This means that $g(\{1,6\})=\{1,6\},$ whence $g^2(1)=1,$ which implies that $g^2=id,$ a contradiction. Having established the claim, consider the action of $G'$ on the factors of length $6$ of the Fibonacci word. One Abelian class is of size five  $\{001001, 001010, 010010, 010100, 100100\}$ and the other of size two $\{100101,101001\}.$ Since $|G'|=3,$ there must be at least two distinct equivalence classes in the first Abelian class, and following the claim, two equivalence classes in the second. Thus at least $4$ equivalence classes combined.

\noindent On the other hand:

\begin{corollary}Let $\sigma \in S_n$ and $G=\langle \sigma \rangle.$ Writing $\sigma=\sigma_1\cdots \sigma_k$ as a product of disjoint cycles, suppose $\gcd(|\sigma_1|,\ldots ,|\sigma_k|)=1.$ Then for every Sturmian word $x$ there exists  $G'\subseteq S_n$ conjugate to $G$ such that $\mbox{Card}\left(\mbox{Fact}_x(n)/\sim_{G'}\right)=\epsilon(G)+1.$
\end{corollary} 

\begin{proof} Since $\gcd(|\sigma_1|,\ldots ,|\sigma_k|)=1,$ we have $G =\langle \sigma_1,\sigma_2,\ldots \sigma_k \rangle.$ Adding if necessary additional $\sigma_i$ of the form $\sigma_i=(a),$ we may assume that
$\sum_{i=1}^k|\sigma_i|=n.$ Let $\{E_1,E_2,\ldots ,E_k\}$ be an interval partition of $\{1,2,\ldots ,n\}$ such that 
$\mbox{Card}(E_i)=|\sigma_i|.$ By Lemma~\ref{multi array}, there exist disjoint cycles $\sigma_1',\sigma_2',\ldots ,\sigma_k'$ 
such that $|\sigma_i|=|\sigma_i'|$ and,  if $G'$ denotes the subgroup of $S_n$ generated by $\sigma'_1,\sigma'_2, \ldots ,\sigma'_k,$ then for all factors $u,v\in \mbox{Fact}_x(n)$ we have $u\sim_G' v$ if and only if $u|_{E_i}\sim_{ab} v|_{E_i}$ for each $1\leq i\leq k.$ Thus $G$ and $G'$ are conjugate in $S_n,$ and by Lemma~\ref{intpart} we have
\[\mbox{Card}(\mbox{Fact}_x(n)/\sim_{G'}) =\mbox{Card}(\mbox{Fact}_x(n)/\sim_k)=k+1=\epsilon(G')+1=\epsilon(G)+1.\]

\end{proof}

 \end{document}